\documentclass[11pt]{amsart}

\usepackage{amsmath,amssymb,amsfonts,amsthm,hyperref,eucal,verbatim,fullpage}

\newtheorem{theorem}{Theorem}[section]

\newtheorem{corollary}[theorem]{Corollary}
\newtheorem{example}[theorem]{Example}

\newtheorem{remark}[theorem]{Remark}

\DeclareMathOperator{\bbC}{\mathbb{C}}

\DeclareMathOperator{\bbR}{\mathbb{R}}

\DeclareMathOperator{\A}{\mathcal{A}}

\DeclareMathOperator{\C}{\mathcal{C}}
\DeclareMathOperator{\D}{\mathcal{D}}

\DeclareMathOperator{\id}{\operatorname{id}}

\DeclareMathOperator{\M}{\mathcal{M}}

\let\slasho=\o
\renewcommand{\o}{\overline}
\renewcommand{\P}{\mathcal{P}}

\DeclareMathOperator{\supp}{\operatorname{supp}}

\begin{document}

\title[Unique Expectations]{Unique Conditional Expectations\\ for Abelian $C^*$-Inclusions}
\author{Vrej Zarikian}
\address{U. S. Naval Academy, Annapolis, MD 21402}
\email{zarikian@usna.edu}
\begin{abstract}
Let $\D \subseteq \A$ be an inclusion of unital abelian $C^*$-algebras. In this note we characterize (in topological terms) when there is a unique conditional expectation $E:\A \to \D$, at least when $\A$ is separable. We also provide the first example of an inclusion with a unique conditional expectation, but multiple pseudo-expectations (in the sense of Pitts).
\end{abstract}
\maketitle

\section{Introduction}

A \emph{$C^*$-inclusion} is a pair $(\C,\D)$, where $\C$ is a unital $C^*$-algebra and $\D \subseteq \C$ is a unital $C^*$-subalgebra (with the same unit). Ideally, the study of $C^*$-inclusions can lead to a better understanding of the containing algebra $\C$, via ``coordinatization'' with respect to the included algebra $\D$. Seminal works in this direction are Kumjian's paper on \emph{$C^*$-diagonals} \cite{Kumjian1986} and Renault's paper on \emph{Cartan subalgebras in $C^*$-algebras} \cite{Renault2008}. These, in turn, were motivated by Feldman and Moore's study of Cartan subalgebras in von Neumann algebras \cite{FeldmanMoore1977}. In particular, in both Kumjian and Renault's settings, there is a unique conditional expectation from $\C$ onto $\D$, which is faithful.\\

A \emph{conditional expectation} from $\C$ onto $\D$ is a contractive linear projection $E:\C \to \D$. Conditional expectations enjoy many nice properties---they are idempotent, unital completely positive (ucp) maps, and are $\D$-bimodular (i.e., $E(d_1xd_2) = d_1E(x)d_2$ for all $x \in \C$, $d_1, d_2 \in \D$) \cite{Tomiyama1957}. A conditional expectation is \emph{faithful} if $E(x^*x) = 0$ implies $x = 0$. Unfortunately, it can easily happen that a $C^*$-inclusion admits no conditional expectations. For example, if $\C$ is injective (in the category \textsf{OpSys} of operator systems and ucp maps), but $\D$ is not, then $(\C,\D)$ admits no conditional expectations. In particular, $(L^\infty[0,1],C[0,1])$ admits no conditional expectations.\\

In \cite{Pitts2012}, Pitts introduced \emph{pseudo-expectations} as a substitute for possibly non-existent conditional expectations. A pseudo-expectation for the $C^*$-inclusion $(\C,\D)$ is a ucp map $\theta:\C \to I(\D)$ such that $\theta|_{\D} = \id$. Here $I(\D)$ is Hamana's \emph{injective envelope} of $\D$, the (essentially) unique minimal injective in \textsf{OpSys} containing $\D$ \cite{Hamana1979}. In fact, $I(\D)$ is a unital $C^*$-algebra containing $\D$ as a unital $C^*$-subalgebra. By injectivity, pseudo-expectations always exist for any $C^*$-inclusion. Every conditional expectation is a pseudo-expectation. Like conditional expectations, pseudo-expectations are contractive and $\D$-bimodular. Unlike conditional expectations, pseudo-expectations are not idempotent. In fact, the composition $\theta \circ \theta$ is generally undefined, since $I(\D)$ rarely sits inside $\C$.\\

In \cite{PittsZarikian2016}, Pitts and the author made a systematic study of pseudo-expectations, with an emphasis on determining which $C^*$-inclusions admit a unique pseudo-expectation. For \emph{abelian $C^*$-inclusions} (i.e., $C^*$-inclusions $(\A,\D)$ with $\A$ abelian), we found an elegant necessary and sufficient condition for there to exist a unique pseudo-expectation. In order to state our result, we remind the reader that by the Gelfand-Naimark Theorem, $\A \cong C(Y)$ (the continuous complex-valued functions on a compact Hausdorff space $Y$), $\D \cong C(X)$, and the inclusion map $\iota:C(X) \to C(Y)$ is determined by a continuous surjection $j:Y \to X$ via the formula $\iota(f) = f \circ j$, $f \in C(X)$.

\begin{theorem}[\cite{PittsZarikian2016}, Cor. 3.21] \label{PsExp}
Let $j:Y \to X$ be a continuous surjection of compact Hausdorff spaces. Then the corresponding abelian $C^*$-inclusion $(C(Y),C(X))$ admits a unique pseudo-expectation if and only if there exists a unique minimal closed set $K \subseteq Y$ such that $j|_K:K \to X$ is a surjection.
\end{theorem}

The main purpose of this note is to find an analogous topological characterization of when $(C(Y),C(X))$ admits a unique conditional expectation. We are able to do so under the restriction that $Y$ is a compact metric space (equivalently, that $C(Y)$ is separable).

\begin{theorem} \label{main}
Let $j:Y \to X$ be a continuous surjection of compact metric spaces. Then the corresponding abelian $C^*$-inclusion $(C(Y),C(X))$ admits a unique conditional expectation if and only if there exists a unique $G_\delta$ set $A \subseteq Y$ such that $j|_A:A \to X$ is an open surjection. In that case, $A \subseteq Y$ is closed and $j|_A:A \to X$ is a homeomorphism.
\end{theorem}

Even though this result appears classical, it is new, to the author's knowledge. Certainly the proof is not classical, since it relies on fairly recent results about \emph{regular averaging operators} and \emph{exact Milutin maps} \cite{RepovsSemenovScepin1997, AgeevTymchatyn2005}.\\

We also produce an abelian $C^*$-inclusion with a unique conditional expectation but multiple pseudo-expectations (Example \ref{Q3}). This answers Question 3 in Section 7.1 of \cite{PittsZarikian2016} affirmatively.\\

\emph{Acknowledgements:} The author would like to thank Ed Tymchatyn for his help in understanding \cite{AgeevTymchatyn2005}. Also, the author is indebted to the anonymous referee for valuable suggestions that improved the clarity of this paper.

\section{Preliminaries}

As mentioned in the introduction, the proof of Theorem \ref{main} consists mainly of appealing to known results from the theory of regular averaging operators and exact Milutin maps. We collect the relevant results here for the reader's convenience, restating them in the form best suited to our purposes. We only include proofs when our restatement appears stronger than the original.\\

For a compact Hausdorff space $Y$ (or for a Polish space $Y$), $C_b(Y)$ will denote the absolutely-bounded continuous complex-valued functions on $Y$ equipped with the supremum norm, $\M(Y)_+ \subseteq C_b(Y)^*$ will denote the finite non-negative regular Borel measures on $Y$, equipped with the weak* topology, and $\P(Y) \subseteq \M(Y)_+$ will denote the regular Borel probability measures on $Y$. For $\mu \in \M(Y)_+$, we define
\[
    \supp(\mu) = \{y \in Y: \text{$\mu(U) > 0$ for every open set $U \subseteq Y$ containing $y$}\}.
\]
It is easy to see that $\supp(\mu) \subseteq Y$ is closed. For $\mu \in \P(Y)$, we have that
\[
    \supp(\mu) = \bigcap\{F: \text{$F \subseteq Y$ is a closed set such that $\mu(F) = 1$}\}.
\]

\begin{theorem}[\cite{Birkhoff1950}] \label{averaging}
Let $j:Y \to X$ be a continuous surjection of compact Hausdorff spaces, with corresponding abelian $C^*$-inclusion $(C(Y),C(X))$.
\begin{enumerate}
\item There exists a bijective correspondence between the conditional expectations $E:C(Y) \to C(X)$ and the continuous maps $\mu:X \to \P(Y)$ such that $\supp(\mu_x) \subseteq j^{-1}(x)$, $x \in X$. Namely, for $g \in C(Y)$ and $x \in X$, $E(g)(x) = \int_Y g(y)d\mu_x(y)$.
\item More generally, there exists a bijective correspondence between the completely positive $C(X)$-bimodule maps $\theta:C(Y) \to C(X)$ and the continuous maps $\nu:X \to \M(Y)_+$ such that $\supp(\nu_x) \subseteq j^{-1}(x)$, $x \in X$.
\end{enumerate}
\end{theorem}

\begin{proof}
We only prove (2). Let $\theta:C(Y) \to C(X)$ be a completely positive $C(X)$-bimodule map. For each $x \in X$, the map $C(Y) \to \bbC: g \mapsto \theta(g)(x)$ is a positive linear functional, and so there exists a unique $\nu_x \in \M(Y)_+$ such that $\theta(g)(x) = \int_Y g(y)d\nu_x(y)$, $g \in C(Y)$. Clearly $\nu:X \to \M(Y)_+$ is continuous.
Since $\theta$ is $C(X)$-bimodular, $\theta(f \circ j) = f\theta(1)$, $f \in C(X)$. Equivalently, $\int_Y f(j(y))d\nu_x(y) = f(x)\nu_x(Y)$, $f \in C(X)$, $x \in X$. Fix $x \in X$. If $y_0 \notin j^{-1}(x)$, then there exists $f \in C(X,[0,1])$ such that $f(j(y_0)) = 1$ and $f(x) = 0$. Let $U = \{y \in Y: f(j(y)) > 1/2\}$, an open set containing $y_0$. Then
\[
    \nu_x(U) = \int_Y \chi_U(y)d\nu_x(y) \leq 2\int_Y f(j(y))d\nu_x(y) = 2f(x)\nu_x(Y) = 0.
\]
Thus $y_0 \notin \supp(\nu_x)$. It follows that $\supp(\nu_x) \subseteq j^{-1}(x)$.

Conversely, let $\nu:X \to \M(Y)_+$ be a continuous map such that $\supp(\nu_x) \subseteq j^{-1}(x)$, $x \in X$. Clearly $\theta:C(Y) \to C(X)$ defined by
\[
    \theta(g)(x) = \int_Y g(y)d\nu_x(y), ~ g \in C(Y), ~ x \in X,
\]
is a completely positive map. If $g \in C(Y)$ and $f \in C(X)$, then for all $x \in X$,
\begin{eqnarray*}
    \theta((f \circ j)g)(x)
    &=& \int_Y f(j(y))g(y)d\nu_x(y) = \int_{j^{-1}(x)} f(j(y))g(y)d\nu_x(y)\\
    &=& \int_{j^{-1}(x)} f(x)g(y)d\nu_x(y) = f(x)\int_{j^{-1}(x)} g(y)d\nu_x(y)\\
    &=& f(x)\int_Y g(y)d\nu_x(y) = f(x)\theta(g)(x),
\end{eqnarray*}
which shows that $\theta$ is $C(X)$-bimodular.
\end{proof}

\begin{remark}
Theorem \ref{averaging} already gives a topological characterization of when there exists a unique conditional expectation $E:C(Y) \to C(X)$. Namely, there exists a unique conditional expectation $E:C(Y) \to C(X)$ if and only if there exists a unique continuous map $\mu:X \to \P(Y)$ such that $\supp(\mu_x) \subseteq j^{-1}(x)$, $x \in X$. In many concrete situations, this is a very practical characterization, easy to apply. Nonetheless, we seek a characterization, Theorem \ref{main}, which is more topological, referring only to $j:Y \to X$, and not requiring knowledge of $\P(Y)$ nor the construction of auxiliary maps.
\end{remark}

\begin{corollary}[\cite{Lloyd1963}] \label{extremal char}
Let $j:Y \to X$ be a continuous surjection of compact Hausdorff spaces, with corresponding abelian $C^*$-inclusion $(C(Y),C(X))$. Let $E:C(Y) \to C(X)$ be a conditional expectation and $\mu:X \to \P(Y)$ be the corresponding continuous map such that $\supp(\mu_x) \subseteq j^{-1}(x)$, $x \in X$. Then the following are equivalent:
\begin{enumerate}
\item $E$ is extremal in the set of all conditional expectations $C(Y) \to C(X)$.
\item $E$ is extremal in the set of all ucp maps $C(Y) \to C(X)$.
\item $E$ is a $*$-homomorphism.
\item There exists a continuous map $\alpha:X \to Y$ such that $j \circ \alpha = \id$ and $E(g) = g \circ \alpha$, $g \in C(Y)$.
\item There exists a continuous map $\alpha:X \to Y$ such that $j \circ \alpha = \id$ and $\mu_x = \delta_{\alpha(x)}$, $x \in X$.
\item $\supp(\mu_x)$ is a singleton for all $x \in X$.
\end{enumerate}
\end{corollary}

\begin{proof}
(1 $\implies$ 2) The set of all conditional expectations $C(Y) \to C(X)$ is a face in the set of all ucp maps $C(Y) \to C(X)$. (For example, see \cite[Prop. 2.4]{PittsZarikian2016}.)

(2 $\implies$ 3) See \cite[Cor. 3.1.6]{Stormer2013}.

(3 $\implies$ 4) Since $E:C(Y) \to C(X)$ is a $*$-homomorphism, there exists a continuous map $\alpha:X \to Y$ such that $E(g) = g \circ \alpha$, $g \in C(Y)$. Since $E:C(Y) \to C(X)$ is a conditional expectation, $f = E(f \circ j) = f \circ j \circ \alpha$, $f \in C(X)$. Thus $j \circ \alpha = \id$.

(4 $\implies$ 5) For all $g \in C(Y)$ and all $x \in X$,
\[
    \int_Y g(y)d\mu_x(y) = E(g)(x) = g(\alpha(x)) = \int_Y g(y)d\delta_{\alpha(x)}(y).
\]

(5 $\implies$ 6) Trivial.

(6 $\implies$ 1) Theorem \ref{averaging} and the fact that for any $y \in Y$, $\delta_y$ is extremal in $\P(Y)$.
\end{proof}

\begin{corollary}[\cite{Lloyd1963}] \label{extremal}
Let $j:Y \to X$ be a continuous surjection of compact Hausdorff spaces, with corresponding abelian $C^*$-inclusion $(C(Y),C(X))$. There exist bijective correspondences between the following sets:
\begin{enumerate}
\item The extremal conditional expectations $E:C(Y) \to C(X)$.
\item The continuous maps $\alpha:X \to Y$ such that $j \circ \alpha = \id$ (i.e., the continuous sections of $j$).
\item The continuous maps $\mu:X \to \P(Y)$ such that $\supp(\mu_x)$ is a singleton contained in $j^{-1}(x)$ for all $x \in X$.
\item The closed sets $K \subseteq Y$ such that $j|_K:K \to X$ is a homeomorphism.
\end{enumerate}
\end{corollary}

\begin{remark} \label{KM}
By Corollary \ref{extremal}, there exists a unique extremal conditional expectation $C(Y) \to C(X)$ if and only if there exists a unique closed set $K \subseteq Y$ such that $j|_K:K \to X$ is a homeomorphism. One might expect that if there exists a unique extremal conditional expectation, then there exists a unique conditional expectation, by appealing to some sort of Krein-Milman Theorem for conditional expectations. If this were true, then one would have the desired topological characterization of when there exists a unique conditional expectation. Unfortunately, as we show in Section \ref{examples} below, one can have few extremal conditional expectations (one or zero) but infinitely many conditional expectations. See also \cite{Lloyd1963}.
\end{remark}

\begin{theorem}[\cite{RepovsSemenovScepin1997}, Thm. 1.2] \label{exact Milutin}
Let $j:Y \to X$ be a continuous open surjection of Polish spaces. Then there exists a continuous map $\mu:X \to \P(Y)$ such that $\supp(\mu_x) = j^{-1}(x)$ for all $x \in X$.
\end{theorem}

\begin{theorem}[\cite{AgeevTymchatyn2005}, Thm. 1.3] \label{inverseSierpinskiHausdorff}
Let $j:Y \to X$ be a continuous surjection of Polish spaces, and let $A \subseteq Y$. If $j|_A:A \to X$ is an open surjection and $(j|_A)^{-1}(x) \subseteq Y$ is closed for all $x \in X$, then $A \subseteq Y$ is $G_\delta$.
\end{theorem}

\begin{theorem}[\cite{AgeevTymchatyn2005}, Thm. 1.2] \label{Milutin}
Let $j:Y \to X$ be a continuous surjection of Polish spaces.
\begin{enumerate}
\item Suppose $A \subseteq Y$ is a $G_\delta$ set such that $j|_A:A \to X$ is an open surjection. Then there exists a continuous map $\mu:X \to \P(Y)$ such that $\supp(\mu_x) \subseteq j^{-1}(x)$ and $\supp(\mu_x) \cap A = (j|_A)^{-1}(x)$, $x \in X$.
\item Conversely, suppose $\mu:X \to \P(Y)$ is a continuous map such that $\supp(\mu_x) \subseteq j^{-1}(x)$. Define $A = \bigcup_{x \in X} \supp(\mu_x)$. Then $A \subseteq Y$ is a $G_\delta$ set such that $j|_A:A \to X$ is an open surjection.
\end{enumerate}
\end{theorem}

\begin{proof}
(1) Since $A \subseteq Y$ is $G_\delta$, it is Polish \cite[Thm. 2.2.1]{Srivastava1998}. By Theorem \ref{exact Milutin}, there exists a continuous map $\nu:X \to \P(A)$ such that $\supp(\nu_x) = (j|_A)^{-1}(x)$, $x \in X$. For $x \in X$ and $B \subseteq Y$ Borel, define $\mu_x(B) = \nu_x(B \cap A)$. Then $\mu_x \in \P(Y)$, and $\mu:X \to \P(Y)$ is continuous. Furthermore, for all $x \in X$,
\[
    \supp(\mu_x) = \o{\supp(\nu_x)} \subseteq j^{-1}(x).
\]
Finally, for all $x \in X$,
\[
    \supp(\mu_x) \cap A = \o{\supp(\nu_x)} \cap A = \supp(\nu_x) = (j|_A)^{-1}(x).
\]

(2) It is easy to see that $j|_A:A \to X$ is a surjection. Furthermore, for all $x \in X$,
\[
    (j|_A)^{-1}(x) = j^{-1}(x) \cap A = \supp(\mu_x) \subseteq Y
\]
is closed. Now let $U \subseteq A$ be open. There exists $V \subseteq Y$ open such that $U = V \cap A$. Suppose $x_i \in j(U)^c$ and $x_i \to x \in X$. Since $j^{-1}(x_i) \cap U = \emptyset$, we have that
\[
    \supp(\mu_{x_i}) \cap V = j^{-1}(x_i) \cap A \cap V = j^{-1}(x_i) \cap U = \emptyset,
\]
which implies $\mu_{x_i}(V) = 0$. Since $\mu_{x_i} \to \mu_x$ weak*, we have that
\[
    \mu_x(V) \leq \liminf_i \mu_{x_i}(V) = 0.
\]
(See \cite[Thm. 6.1]{Parthasarathy1967}.) Thus $\supp(\mu_x) \cap V = \emptyset$, which implies
\[
    j^{-1}(x) \cap U = j^{-1}(x) \cap A \cap V = \supp(\mu_x) \cap V = \emptyset.
\]
Hence $x \in j(U)^c$, which implies $j(U)^c$ is closed, which in turn implies $j(U)$ is open. It follows that $j|_A:A \to X$ is open. By Theorem \ref{inverseSierpinskiHausdorff}, $A \subseteq Y$ is $G_\delta$.
\end{proof}

\section{Proof of the Main Result}

\begin{proof}[Proof of Theorem \ref{main}.]
($\Rightarrow$) Suppose there exists a unique conditional expectation $E:C(Y) \to C(X)$. Then (trivially) $E$ is an extremal conditional expectation. By Theorem \ref{averaging}, there exists a unique continuous map $\mu:X \to \P(Y)$ such that $\supp(\mu_x) \subseteq j^{-1}(x)$, $x \in X$. By Corollary \ref{extremal char}, there exists a continuous map $\alpha:X \to Y$ such that $j \circ \alpha = \id$ and $\mu_x = \delta_{\alpha(x)}$, $x \in X$. Clearly $\alpha(X) \subseteq Y$ is closed and $j|_{\alpha(X)}:\alpha(X) \to X$ is a homeomorphism. Now let $A \subseteq Y$ be a $G_\delta$ set such that $j|_A:A \to X$ is an open surjection. By Theorem \ref{Milutin}, part (1), there exists a continuous map $\nu:X \to \P(Y)$ such that $\supp(\nu_x) \subseteq j^{-1}(x)$ and $\supp(\nu_x) \cap A = (j|_A)^{-1}(x)$, $x \in X$. By uniqueness, $\nu = \mu$. Thus for all $x \in X$,
\[
    \emptyset \neq (j|_A)^{-1}(x) = j^{-1}(x) \cap A = \supp(\nu_x) \cap A = \supp(\mu_x) \cap A \subseteq \{\alpha(x)\}.
\]
It follows that $j^{-1}(x) \cap A = \{\alpha(x)\}$ for all $x \in X$, so that $A = \alpha(X)$.

($\Leftarrow$) Conversely, suppose there exists a unique $G_\delta$ set $A \subseteq Y$ such that $j|_A:A \to X$ is an open surjection.  We claim that $j|_A:A \to X$ is injective. Indeed, if $y_1, y_2 \in A$, with $y_1 \neq y_2$ and $j(y_1) = j(y_2)$, then $B = A \backslash \{y_2\} \subseteq Y$ is $G_\delta$ and $j|_B:B \to X$ is an open surjection, contradicting the uniqueness of $A$. It follows from the claim that $j|_A:A \to X$ is a homeomorphism. Set $\alpha = (j|_A)^{-1}:X \to A$. Then $\mu:X \to \P(Y)$ defined by $\mu_x = \delta_{\alpha(x)}$, $x \in X$, is a continuous map such that $\supp(\mu_x) \subseteq j^{-1}(x)$, $x \in X$. Now let $\nu:X \to \P(Y)$ be a continuous map such that $\supp(\nu_x) \subseteq j^{-1}(x)$, $x \in X$. Define $C = \bigcup_{x \in X} \supp(\nu_x)$. By Theorem \ref{Milutin}, part (2), $C \subseteq Y$ is a $G_\delta$ set such that $j|_C:C \to X$ is an open surjection. By uniqueness, $C = A$. Thus for all $x \in X$,
\[
    \supp(\nu_x) = j^{-1}(x) \cap C = j^{-1}(x) \cap A = (j|_A)^{-1}(x) = \{\alpha(x)\}.
\]
It follows that $\nu_x = \delta_{\alpha(x)} = \mu_x$, $x \in X$, so that $\nu = \mu$. By Theorem \ref{averaging}, there exists a unique conditional expectation $E:C(Y) \to C(X)$.
\end{proof}

\section{Examples} \label{examples}

In this section we use Theorems \ref{main} and \ref{averaging} and Corollary \ref{extremal} to analyze the conditional expectations for various interesting abelian $C^*$-inclusions. The first two examples illustrate Remark \ref{KM}, that in general there is no Krein-Milman Theorem for conditional expectations. This was initially observed in \cite{Lloyd1963}. The last two examples combine to show that there exists an abelian $C^*$-inclusion which admits a unique conditional expectation but infinitely many pseudo-expectations. To our knowledge, this is the first such example, and it answers Question 3 from Section 7.1 of \cite{PittsZarikian2016} affirmatively.

\begin{example} \label{canonical}
Let
\[
    Y = ([0,2] \times \{0\}) \cup ([0,1] \times \{1\}) \subseteq \bbR^2,
\]
\[
    X = [0,2] \subseteq \bbR,
\]
and
\[
    j:Y \to X:(x,y) \mapsto x.
\]
Then the corresponding abelian $C^*$-inclusion $(C(Y),C(X))$ has infinitely many conditional expectations, only one of which is extremal.
\end{example}

\begin{proof}
It is easy to see that there exists a unique closed set $K \subseteq Y$ such that $j|_K:K \to X$ is a homeomorphism, namely $K = [0,2] \times \{0\}$. But $K$ is not the unique $G_\delta$ set $A \subseteq Y$ such that $j|_A:A \to X$ is an open surjection, since $L = Y \backslash \{(1,1)\}$ is another such set. The result now follows by Theorem \ref{main} and Corollary \ref{extremal}.

Alternatively, $\mu:X \to \P(Y)$ defined by $\mu_x = \delta_{(x,0)}$ is the unique continuous function such that $\supp(\mu_x)$ is a singleton contained in $j^{-1}(x)$ for all $x \in X$. But $\nu:X \to \P(Y)$ defined by
\[
    \nu_x = \begin{cases}
        x\delta_{(x,0)} + (1-x)\delta_{(x,1)}, & 0 \leq x \leq 1\\
        \delta_{(x,0)}, & 1 \leq x \leq 2
    \end{cases}
\]
is another continuous map such that $\supp(\nu_x) \subseteq j^{-1}(x)$, $x \in X$. So the result also follows by Theorem \ref{averaging} and Corollary \ref{extremal}.
\end{proof}

\begin{example} \label{space-filling}
Let $I = [0,1]$ be the unit interval and $I^2 = I \times I$ be the unit square. By \cite[Thm. 2]{Taskinen1993}, there exists a continuous surjection $j:I \to I^2$ (i.e., a space-filling curve) such that the corresponding abelian $C^*$-inclusion $(C(I),C(I^2))$ admits a conditional expectation. Then $(C(I),C(I^2))$ admits infinitely many conditional expectations, none of which are extremal.
\end{example}

\begin{proof}
Suppose there exists an extremal conditional expectation $E:C(I) \to C(I^2)$. Then by Corollary \ref{extremal} there exists a continuous map $\alpha:I^2 \to I$ such that $j \circ \alpha = \id$. In particular, $\alpha$ is injective. Let $z_1, z_2, z_3 \in I^2$ be such that $\alpha(z_1) < \alpha(z_2) < \alpha(z_3)$. Then $I^2 \backslash \{z_2\}$ is connected but $\alpha(I^2 \backslash \{z_2\})$ is not, a contradiction.
\end{proof}

\begin{example} \label{Cantor}
Let $\C = \prod_{n=1}^\infty \{0,1\}$ be the Cantor set, $I = [0,1]$ be the unit interval, and $j:\C \to I$ be the continuous surjection defined by
\[
    j((a_n)) = \sum_{n=1}^\infty \frac{a_n}{2^n}.
\]
Then the corresponding abelian $C^*$-inclusion $(C(\C),C(I))$ admits no conditional expectations. In fact, the only completely positive $C(I)$-bimodule map $\theta:C(\C) \to C(I)$ is $\theta = 0$.
\end{example}

\begin{proof}
We denote by $D \subseteq I$ the dyadic rationals (excluding $0$ and $1$). Let $x \in I$. If $x \in I \backslash D$, let $\alpha_0(x)$ and $\alpha_1(x)$ both equal to the unique pre-image of $x$ under $j$. Otherwise, if $x \in D$, let $\alpha_0(x)$ be the pre-image of $x$ under $j$ which ends with an infinite string of zeros, and let $\alpha_1(x)$ be the pre-image of $x$ under $j$ which ends with an infinite string of ones. Now let $\mu:I \to \M(\C)_+$ be a continuous map such that $\supp(\mu_x) \subseteq j^{-1}(x)$, $x \in I$. Then there exist functions $\lambda_0, \lambda_1:I \to [0,\infty)$ such that
\[
    \mu_x = \lambda_0(x)\delta_{\alpha_0(x)} + \lambda_1(x)\delta_{\alpha_1(x)}, ~ x \in X.
\]
Suppose $x \in D$. Let $\{x_n\} \subseteq I \backslash D$ be such that $\alpha_0(x_n) \to \alpha_0(x)$. Then $x_n = j(\alpha_0(x_n)) \to j(\alpha_0(x)) = x$, and so $\mu_{x_n} \to \mu_x$ weak*. Now let $U, V \subseteq \C$ be disjoint open sets such that $\alpha_0(x) \in U$ and $\alpha_1(x) \in V$. Then
\[
    \lambda_1(x) = \mu_x(V) \leq \liminf_n \mu_{x_n}(V) = 0.
\]
Likewise $\lambda_0(x) = 0$. Thus $\mu_x = 0$. Since $D$ is dense in $I$, $\mu = 0$. The result now follows by Theorem \ref{averaging}.
\end{proof}

\begin{example} \label{Q3}
Let $(\A,\D)$ be an abelian $C^*$-inclusion such that the only completely positive $\D$-bimodule map $\theta:\A \to \D$ is $\theta = 0$. (For example, the abelian $C^*$-inclusion in Example \ref{Cantor} above.) Let $\tilde{\D} = \{d \oplus d: d \in \D\}$. Then the abelian $C^*$-inclusion $(\A \oplus \D,\tilde{\D})$ admits a unique conditional expectation but infinitely many pseudo-expectations.
\end{example}

\begin{proof}
The map $E:\A \oplus \D \to \tilde{\D}: x \oplus d \mapsto d \oplus d$ is clearly a conditional expectation. We claim that it is the only one. Indeed, let $\Theta:\A \oplus \D \to \tilde{\D}$ be a conditional expectation. For $x \in \A$, define $\theta(x) = \pi(\Theta(x \oplus 0))$, where $\pi:\tilde{\D} \to \D$ is given by $\pi(d \oplus d) = d$, $d \in \D$. It is easy to see that $\theta:\A \to \D$ is a completely positive $\D$-bimodule map, so that $\theta = 0$ by hypothesis. Then $\Theta(x \oplus 0) = 0 \oplus 0$, $x \in \A$, so that $\Theta(x \oplus d) = d \oplus d$, $x \oplus d \in \A \oplus \D$.

To show that $(\A \oplus \D,\tilde{\D})$ admits infinitely many pseudo-expectations, it suffices to construct a pseudo-expectation $\tilde{\theta}:\A \oplus \D \to I(\tilde{\D})$ distinct from the conditional expectation $E$ above, because any convex combination of pseudo-expectations is again a pseudo-expectation. To that end, let $\theta:\A \to I(\D)$ be a pseudo-expectation and define $\tilde{\theta}:\A \oplus \D \to \widetilde{I(\D)}: x \oplus d \mapsto \theta(x) \oplus \theta(x)$. (Here $\widetilde{I(\D)} = \{v \oplus v: v \in I(\D)\} = I(\tilde{\D})$.) It is easy to see that $\tilde{\theta}$ is a pseudo-expectation, and $\tilde{\theta}(1 \oplus 0) = 1 \oplus 1 \neq 0 \oplus 0 = E(1 \oplus 0)$.
\end{proof}

\end{document}